\documentclass[oneside,reqno,english]{amsart}
\usepackage[T1]{fontenc}
\usepackage[latin9]{inputenc}
\usepackage{amsthm,amssymb,amstext}
\usepackage{enumerate}

\makeatletter

\newtheorem{thm}{\protect\theoremname}[section]
  \theoremstyle{definition}
  \newtheorem{defn}[thm]{\protect\definitionname}
  \theoremstyle{remark}
  \newtheorem{rem}[thm]{\protect\remarkname}
  \theoremstyle{plain}
  \newtheorem{prop}[thm]{\protect\propositionname}
  \theoremstyle{plain}
  \newtheorem{lem}[thm]{\protect\lemmaname}
  \theoremstyle{plain}
  \newtheorem{cor}[thm]{\protect\corollaryname}

\makeatother

\usepackage{babel}
  \providecommand{\definitionname}{Definition}
  \providecommand{\lemmaname}{Lemma}
  \providecommand{\propositionname}{Proposition}
  \providecommand{\remarkname}{Remark}
\providecommand{\theoremname}{Theorem}
\providecommand{\corollaryname}{Corollary}

\newcommand{\const}{\mathrm{const\,}}
\newcommand{\vol}{\mathrm{vol}}

\newcommand{\loc}{\mathrm{loc}}

\begin{document}

\title[The~Sobolev--Poincar\'e Inequality and the~$L_{q,p}$-Cohomology]
{The~Sobolev--Poincar\'e Inequality and \\
the~$L_{q,p}$-Cohomology of~Twisted Cylinders}


\author{Vladimir Gol$'$dshtein}
\address{Department of Mathematics,
Ben Gurion University of the Negev,
P.O.Box 653, Beer Sheva, Israel} 
\email{vladimir@bgumail.bgu.ac.il}

\author{Yaroslav Kopylov}
\address{Sobolev Institute of Mathematics, Pr.~Akad. Koptyuga 4,
630090, Novosibirsk, Russia  \\
Novosibirsk State University, ul.~Pirogova~1,
630090, Novosibirsk, Russia}
\email{yakop@math.nsc.ru}

\thanks{The second author was supported by the Program of Basic Scientific Research 
of the Siberian Branch of the Russian Academy of Sciences.}

\begin{abstract}
We establish a~vanishing result for the~$L_{q,p}$-cohomology ($q\ge p$) of~a~twisted 
cylinder, which is a~generalization of~a~warped cylinder. The~result is new even 
for warped~cylinders. We base on~the~methods for proving 
the~$(p,q)$~Sobolev--Poincar\'e inequality developed by~L.~Shartser. 

\textit{Mathematics Subject Classification.} 58A12, 46E30, 55N05.

\textit{Key words and phrases}: differential form, $L_{q,p}$-cohomology,
twisted cylinder, homotopy operator 
\end{abstract}
\maketitle

\section{Introduction}

The \textit{$L_{q,p}$-cohomology} $H_{q,p}^{k}(M)$ of a Riemannian manifold $(M,g)$ is, 
by~definition, the quotient of the space of closed $p$-integrable differential $k$-forms 
by the exterior differentials of $q$-integrable $k$-forms. 
If $p=q$ then $L_{q,p}$-cohomology is usually referred to simply as $L_p$-cohomology 
and the index~$p$ is used instead of~$p,p$ in all the notations. 

A \textit{twisted product} $X\times_{h}Y$ of two Riemannian manifolds 
$(X,g_{X})$ and $(Y,g_{Y})$ is the direct product manifold $X\times_{g}Y$ endowed
with a~Riemannian metric of the form 
\begin{equation}
g:=g_{X}+h^{2}(x,y)g_{Y},\label{tp-metric}
\end{equation}
where $h:X\times Y\to\mathbb{R}$ is a~smooth positive function (see
\cite{Chen81}). If $X$ is a~half-interval $[a,b)$ then the twisted product $X\times_{h}Y$ 
is called a~{\it twisted cylinder}.
                                                                                                    
We refer to an $m$-dimensional Riemannian manifold $(M,g_M)$ as 
an~\textit{asymptotic twisted product} (respectively, as an {\it asymptotic twisted cylinder})
if, outside an $m$-dimensional compact submanifold,
it is bi-Lipschitz equivalent to a~twisted product (respectively, to a~twisted cylinder). 

In this paper, we prove some vanishing results for the $L_{q,p}$-cohomology
of twisted cylinders $[a,b)\times_{h}N$ for a~positive smooth function
$h:[a,b)\times N\to\mathbb{R}$ in the case where the base $N$ is a closed manifold and 
$p\ge q>1$, $\frac{1}{p}-\frac{1}{q}<\frac{q-1}{q(\dim N+1)}$. 

If in~(\ref{tp-metric}) the function~$h$ depends only on~$x$ then we obtain the familiar notion 
of a~\textit{warped product} (see \cite{BiON}). Twisted products were the object of~recent 
investigations~\cite{BDDO2012,DBZ2012,Fa2013,FGKU2001,GKop16,KJKP2005,PoRe93}.
The $L_{q,p}$-cohomology of warped cylinders $[a,b)\times_{h}N$,
i.e., of product manifolds $[a,b)\times N$ endowed with a warped product
metric 
$$
g=dt^{2}+h^{2}(t)g_{N},
$$
where $g_{N}$ is the Riemannian metric of~$N$ and $h:[a,b)\to\mathbb{R}$
is a positive smooth function, was studied by Gol$'$dshtein, Kuz$'$minov,
and Shvedov \cite{GKSh90_2}, Kuz$'$minov and Shvedov \cite{KSh93,KSh96}
(for $p=q$), and Kopylov \cite{Kop07} for $p,q\in[1,\infty)$, 
$\frac{1}{p}-\frac{1}{q}<\frac{1}{\dim\,N+1}$.

The main result of the paper (Theorem~\ref{thm: main global}) 
states that the $L_{q,p}$-co\-ho\-mo\-lo\-gy 
$H^k_{q,p}(C_{a,b}^h N)$ of~the~twisted cylinder~$C_{a,b}^h N$ with~$q\ge p\geq 1$ and 
$\frac{1}{p}-\frac{1}{q}<\frac{q-1}{q(\dim N+1)}$
is zero provided that the~de~Rham cohomology $H^k_{\mathrm{DR}}(N)$ of~the~base~$N$
is trivial and some integral conditions on the~twisting function involving~$p$, $q$ and
an~auxiliary parameter~$\overline{p}$ are fulfilled.
\smallskip                                                                                   

The~paper is organized as follows: In~Sec.~\ref{bd}, we recall some basic definitions 
concerning the~$L_{q,p}$-cohomology of~Riemannian manifolds. 
Sec.~\ref{diff_twis} describes the~representations 
of~differential forms on~a~twisted cylinder obtained
in~\cite{GKop16} and analogous to~the~representations of~forms on~a~warped product
proposed by~Gol$'$dshtein, Kuz$'$minov, and Shvedov in~\cite{GKSh90_1}. 
In~Sec.~\ref{weigh-sobp},
we develop a~version of~the~weighted Sobolev--Poncar\'e inequality for convex sets
in~$\mathbb{R}^n$ by~introducing a~homotopy operator and consider some of its 
consequences; the~exposition is based on~the~ideas of~Shartser suggested 
in~\cite{SharThes} and~\cite{Shar2011}. In~Sec.~\ref{new-homot}, 
we consider a~new homotopy 
operator~$A_\alpha$ on~differential forms defined on~a~convex domain in~$\mathbb{R}^n$
and show that it guarantees the~fulfillment of~an~inequality of~Sobolev--Poincar\'e-type
for~$q\ge p\ge 1$ and $\frac{1}{p}-\frac{1}{q}<\frac{1}{n}$. In~Sec.~\ref{global}, 
using the~ideas of~Shartser's article~\cite{Shar2011},
we ``glue'' local homotopy operators on~a~twisted cylinder to~obtain
a~global homotopy operator. In~Sec.~\ref{rel-cohom}, we use this global homotopy 
operator for~proving our above-metioned main result 
on~the~triviality of~the~$L_{q,p}$-cohomology 
of~a~twisted cylinder (Theorem~\ref{thm: main global}), and in~Sec.~\ref{cohom-asymp},
we extend this theorem to~asymptotic twisted  cylinders (Theorem~\ref{thm:app-abs}). 
Sec.~9 contains some examples.

\section{Basic Definitions}\label{bd}

We recall the main definitions and notations. 

Below we tacitly assume all manifolds to be oriented.

Let $M$ be a smooth oriented Riemannian manifold. 
Denote by $\mathcal{D}^k(M):=C_{0}^{\infty}(M,\Lambda^k)$
the space of all smooth differential $k$-forms with compact support
contained in~$M\setminus\partial M$ denote by $L_{loc}^{1}(M,\Lambda^k)$ 
the space of locally integrable differential forms.
                                                                                             
Denote by $L^{p}(M,\Lambda^{k})$ the Banach space
of locally integrable differential $k$-forms endowed with the norm
$\|\theta\|_{L^p(M,\Lambda^k)}:=\left(\int_{M}|\theta|^{p}dx\right)^{\frac{1}{p}}<\infty$
(as usual, we identify forms coinciding outside a~set of measure zero).
Of~course, we can add a~positive (smooth) weight~$\sigma:M\to \mathbb{R}$ and thus 
integrate $|\theta|^p \sigma^p$ to~obtain the~weighted $L^p$-space 
$L^{p}(M,\Lambda^{k},\sigma)$.

\begin{defn}
We call a~differential $(k+1)$-form
$\theta\in L^{1}_{loc}(M,\Lambda^{k+1})$ \emph{the weak exterior derivative}
(or \emph{differential}) of a differential $k$-form $\phi\in L^{1}_{\loc}(M,\Lambda^{k})$ 
and write $d\phi=\theta$ if 
\[
\int_{M}\theta\wedge\omega=(-1)^{k+1}\int_{M}\phi\wedge d\omega
\]
for any $\omega\in\mathcal{D}^{n-k}(M)$. \end{defn}

\begin{rem}
Note that the orientability of $M$ is not substantial in~this definition since one may
take integrals over orientable domains on $M$ instead of integrals \linebreak
over~$M$. 
\end{rem}
We then introduce an analog of Sobolev spaces for differential $k$-forms,
i.e., the space of $q$-integrable forms with $p$-integrable weak exterior
derivative: 
$$
\Omega_{q,p}^{k}(M)=\left\{ \,\omega\in L^{q}(M,\Lambda^{k})\;|\, 
d\omega\in L^{p}(M,\Lambda^{k+1})\right\}.
$$
This is a Banach space for the graph norm 
\[
\|\omega\|_{q,p}=\left(\|\omega\|_{L^q(M,\Lambda^k)}^{2}
+\|d\omega\|_{L^p(M,\Lambda^{k+1})}^{2}\right)^{1/2}.
\]
The space $\Omega_{q,p}^{k}(M)$ is a reflexive Banach space for
any $1<q,p<\infty$. This can be proved using standard arguments of
functional analysis.



We now define our basic ingredients (for three parameters $r,q,p$).
\begin{defn}
Put

(a) $Z_{p,r}^{k}(M)=\mathrm{Ker}[d:\Omega_{p,r}^{k}(M)\to L^{r}(M,\Lambda^{k+1})]$.

(b) $B_{q,p}^{k}(M)=\mathrm{Im}[d:\Omega_{q,p}^{k-1}(M)\to L^{p}(M,\Lambda^{k})]$.


\end{defn}

The subspace $Z_{p,r}^{k}(M)$ does not depend on $r$ and is a closed
subspace in $L^{p}(M,\Lambda^{k})$ (see Lemma~\cite[Lemma~2.4(i)]{GT2012}).
This allows us to use the notation $Z_{p}^{k}(M)$ for all $Z_{p,r}^{k}(M)$.
Note that $Z_{p}^{k}(M)\subset L^{p}(M,\Lambda^{k})$ is always a
closed subspace but that is in general not true for $B_{q,p}^{k}(M)$.
Denote by $\overline{B}_{q,p}^{k}(M)$ its closure in the $L^{p}$-topology.
Observe also that since $d\circ d=0$,
one has $\overline{B}_{q,p}^{k}(M)\subset Z_{p}^{k}(M)$. Thus, 
\[
B_{q,p}^{k}(M)\subset\overline{B}_{q,p}^{k}(M)\subset Z_{p}^{k}(M)
=\overline{Z}_{p}^{k}(M)\subset L^{p}(M,\Lambda^{k}).
\]

\begin{defn}
Suppose that $1\leq q,p\leq\infty$. The \emph{$L_{q,p}$-cohomology}
of $(M,g)$ is defined as the quotient 
\[
H_{q,p}^{k}(M):=Z_{p}^{k}(M)/B_{q,p}^{k}(M)\,,
\]
 and the \emph{reduced $L_{q,p}$-cohomology} of $(M,g)$ is, by definition, 
the space
\[
\overline{H}_{q,p}^{k}(M):=Z_{p}^{k}(M)/\overline{B}_{q,p}^{k}(M)\,.
\]
 
\end{defn}

Since $B_{p,q}^{k}$ is not always closed, the $L_{p}$-cohomology
is in general a (non-Hausdorff) semi-normed space, while the reduced
$L_{p}$-cohomology is a Banach space. 

Below $|X|$ stands for the volume of a Riemannian manifold $(X,g)$.

\medskip

It follows from the results of~\cite{GT2006} that, under suitable assumptions on
$p,q$, the $L_{q,p}$-cohomology of a Riemannian manifold~$M$ can be expressed 
in terms of smooth forms.

Let $C^\infty(M,\Lambda^k)$ be the~space of~smooth $k$-forms on~$M$.

Introduce the notations:
\begin{gather*}                          
C^\infty L^p(M,\Lambda^k):=C^\infty(M,\Lambda^k)\cap L^p(M,\Lambda^k);  \\
C^\infty L^p(M,\Lambda^k,\sigma):=C^\infty(M,\Lambda^k)\cap L^p(M,\Lambda^k,\sigma);\\
C^\infty\Omega_{q,p}^{k}(M):=C^\infty(M,\Lambda^k)\cap\Omega_{q,p}^{k}(M);  \\
C^\infty H_{q,p}^{k}(M)
:= \frac{C^\infty(M,\Lambda^k)\cap Z_p^k(M)}{C^\infty(M,\Lambda^k)\cap B_{q,p}^{k}(M)}; \\
C^\infty \overline{H}_{q,p}^{k}(M)
:= \frac{C^\infty(M,\Lambda^k)\cap Z_p^k(M)}
{C^\infty(M,\Lambda^k)\cap\overline{B}_{q,p}^{k}(M)}.
\end{gather*}

\begin{thm}\label{sm-cohom} 
\emph{\cite[Theorem~12.5 and~12.8, Corollary~12.9]{GT2006}.}
Let $(M,g)$ be a $n$-dimensional
Riemannian manifold and suppose the~fulfillment of~one of~the~following
conditions:  

$\bullet$ $p,q\in (1,\infty)$ and $\frac{1}{p}-\frac{1}{q} \leq \frac{1}{n}$;

$\bullet$ $p,q\in [1,\infty)$ and $\frac{1}{p}-\frac{1}{q} < \frac{1}{n}$. 

Then the cohomology $H^*_{q,p}(M)$ can be represented by smooth forms, and thus
$H^*_{q,p}(M)=C^\infty H_{q,p}^*(M)$.

More exactly, any closed form in $Z^k_p(M)$ is cohomologous to a
smooth form in~$L^p(M)$. Furthermore, if two smooth closed forms
$\alpha,\beta \in C^{\infty}(M,\Lambda^k)\cap Z^k_p(M)$ are cohomologous
modulo $d\Omega_{q,p}^{k-1}(M)$ then they are cohomologous modulo
$dC^{\infty}\Omega_{q,p}^{k-1}(M)$.

Similarly, any reduced cohomology class can be represented by~a~smooth form.
\end{thm}

\section{Differential Forms on a Twisted Cylinder}\label{diff_twis}

From now on, $C_{a,b}^h N$ is the twisted cylinder $[a,b)\times_h N$, 
that is, the product of~a~half-interval $[a,b)$ and a~closed smooth 
$n$-dimensional Riemannian manifold 
$(N, g_N)$ equipped with the Riemannian metric $dt^2+ h^2(t,x) g_N$, where 
$h: [a,b)\times N \to \mathbb{R}$ is a~smooth positive function.

Every differential form on~$[a,b)\times N$ admits a unique representation of the form
$\omega=\omega_{A}+dt\wedge\omega_{B}$,
where the forms $\omega_{0}$ and $\omega_{1}$ do not contain $dt$ (cf. \cite{GKSh90_1}).
It means that $\omega_{0}$ and $\omega_{1}$ can be viewed
as one-parameter families $\omega_{A}(t)$ and $\omega_{B}(t)$, $t\in I$,
of differential forms on $N$. 

The modulus of a form $\omega$ of degree $k$ on $C_{a,b}^h N$ is expressed
via the moduli of $\omega_{A}(t)$ and $\omega_{B}(t)$ on $N$ as
follows: 
\begin{equation}\label{eq:module}         
|\omega(t,x)|_{C_{a,b}^h N}=\bigl[h^{-2k}(t,x) |\omega_{A}(t,x)|_{N}^{2} + h^{-2(k+1)}(t,x) 
|\omega_{B}(t,x)|_{N}^{2}\bigr]^{1/2}
\end{equation}

Consequently, 
\begin{multline}\label{eq:norm}
\|\omega\|_{L^p(C_{a,b}^h N,\Lambda^k)} \\
=\left[\!\int_{a}^{b}\int_{N}\bigl(h^{2(\frac{n}{p}-k)}(t,x) |\omega_{A}(t,x)|_{N}^{2}\!
+ h^{2(\frac{n}{p}-k+1)}(t,x) |\omega_{B}(t,x)|_{N}^{2}\bigr)^{\frac{p}{2}} 
dxdt\!\right]^{\frac{1}{p}}.
\end{multline}

Put 
\[
f_{k,p}(t)=\min_{x\in N} \bigl\{ h^{\frac{n}{p}-k}(t,x) \bigr\}
\]
 and 
\[
F_{k,p}(t)=\max_{x\in N} \bigl\{ h^{\frac{n}{p}-k}(t,x) \bigr\}.
\]

\section{The Weighted Sobolev--Poincare Inequality for Convex Sets in~$\mathbb{R}^n$}
\label{weigh-sobp}

Denote by $\Omega_{loc}^{*}(M)$ the space  all locally integrable
differential forms with locally integrable weak differential. 

Suppose that $D\subset R^{n}$ is a convex set and $\psi_{y}:D\times[0,1]\to D$,
$\psi_{y}(x,t):=tx+(1-t)y$, is the homotopy induced by the convex structure.
For a $k$-form $\omega\in\Omega_{loc}^{k}(D)$ the pullback $\psi{}_{y}^{*}\omega$
can be written in the form 
\[
\psi{}_{y}^{*}\omega(x,t)
=\left(\psi{}_{y}^{*}\omega\right)_{0}(x,t)+dt\wedge\left(\psi{}_{y}^{*}\omega\right)_{1}(x,t),
\]
where $\left(\psi{}_{y}^{*}\omega\right)_{0}$ and $\left(\psi{}_{y}^{*}\omega\right)_{1}$
do not contain~$dt$.

For each $y\in D$ define a homotopy operator 
\[
K_y:\Omega_{loc}^{k}(D)\to\Omega_{loc}^{k-1}(D)
\]
as follows:
\[
K_y\omega(x):=\int_{0}^{1}\left(\psi{}_{y}^{*}\omega\right)_{1}(t)\, dt
\]

It is easy to~see that $K_y$ takes smooth forms to~smooth forms.
It is proved in \cite{IL93} that $K_{y}d\omega+dK_{y}\omega=\omega$
The following proposition is a generalization of results from~\cite{BoMi95} and
Shartser's thesis \cite{SharThes} (see also \cite{Shar2011}) to the weighted case and 
to unbounded convex domains.

\begin{prop}\label{bounded}
Suppose that $D$ is a~convex set in $\mathbb{R}^n$, $q\geq p\geq1$,  
and $\beta:D\to\mathbb{R}$ is a~positive smooth function. 

If the inequality 
\[
C(k,p,q,n,\beta)
:=\int_{0}^{1}  
\sup_{z\in D} \| \beta(x) \mathbf{1}_{tx+(1-t)D}(z) \|_{L^q(D,dx)}
t^k (1-t)^{-n/p} dt < \infty
\]
holds then the inequality
\[
\left\| \beta(x)\left\| \frac{K_{y}d\omega(x)}
{\left|x-y\right|}\right\|_{L^{p}(D,dy)}\right\| _{L^{q}(D,dx)}
\leq C(k,p,q,n,\beta)\left\| d\omega\right\| _{L^{p}(D,\Lambda^{k+1})}.
\]
is valid for every $\omega\in\Omega_{loc}^{k}(D)$ such that $d\omega\in L^p(D,\Lambda^{k+1})$.
Here $\mathbf{1}_{xt+(1-t)D}$ is the characteristic function of the set $xt+(1-t)D$.
\end{prop}

\begin{proof}
By the definition of $K_{y},$ we have
\begin{multline*}
\hspace{-3mm}
\left\| \beta(x)\! \left\| \frac{K_{y}d\omega(x)}{\left|x-y\right|}\right\|_{L^{p}(D,dy)}
\right\| _{L^{q}(D,dx)} \!\!\!\!
=\! \left\| \beta(x)\! \left\| 
\int_{0}^{1} \!\!
\frac{\left(\psi{}_{y}^{*}d\omega\right)_{1}(x,t)}{\left|x-y\right|}dt\right
\|_{L^{p}(D,dy)}\right\| _{L^{q}(D,dx)}
\\
\leq \int_{0}^{1}\left\| \beta(x)\left\| \frac{\left(\psi{}_{y}^{*}d\omega\right)_{1}(x,t)}
{\left|x-y\right|}\right\| _{L^{p}(D,dy)}\right\| _{L^{q}(D,dx)}dt
\\
\le \int_{0}^{1}\left\{ \int_{D}\beta^q(x)\left[\int_{D}\frac{\left|\left(\psi{}_{y}^{*}d\omega\right)_{1}
(x,t)\right|^{p}}{\left|x-y\right|^{p}}dy\right]^{q/p}dx\right\} ^{1/q}dt.
\end{multline*}

As usual, we identify the tangent space to~$\mathbb{R}^n$ at any of its points with $\mathbb{R}^n.$

By easy calculations, 
\[
\left|\left(\psi_{y}^{*}d\omega\right)_{1}(x,t)\right|
\leq\left|x-y\right|t^{k}\left|d\omega\left(\psi_{y}(x,t)\right)\right|.
\]
Therefore, 
\begin{multline*}
\int_{0}^{1}\left\{ \int_{D} \beta^q(x)
\left[\int_{D}\frac{\left|\left(\psi_{y}^{*}d\omega\right)_{1}
(x,t)\right|^{p}}{\left|x-y\right|^{p}}dy\right]^{q/p}dx\right\} ^{1/q}dt
\\
\leq \int_{0}^{1}\left\{ \int_{D} \beta^q(x)\left[\int_{D}t^{kp}
\left|d\omega\left(\psi_{y}(x,t)\right)\right|^{p}dy\right]^{q/p}dx\right\} ^{1/q} dt \\
= \int_{0}^{1}\left\{ \int_{D}\beta^q(x)\left[\int_{D}t^{kp}
\left|d\omega\left(tx+(1-t)y\right)\right|^{p}dy\right]^{q/p}dx\right\}^{1/q} dt :=I.
\end{multline*}
The change of variables $z=tx+(1-t)y$ in~the~inner integral yields
\[
I=\int_{0}^{1}\left\{ \int_{D}\beta^q(x) \left[\int_{tx+(1-t)D}
\left|d\omega(z)\right|^{p} dz \right]^{q/p} dx\right\} ^{1/q} t^{k}(1-t)^{-n/p}dt
\]

Since $D$ is convex, the set $tx+(1-t)D$ is contained in~$D$ for all $x\in D$ and
$t\in [0,1]$. Using Minkowski's integral inequality, we infer
\begin{align*}
& \left\{\int_D \beta^q(x) 
 \left[ \int_{tx+(1-t)D}
\left|d\omega (z)\right|^p dz\right]^{q/p} dx\right\}^{1/q}
\\
& =\left\{ \int_{D}\beta^q(x)
\left[\int_{D} \mathbf{1}_{tx+(1-t)D}(z) \left|d\omega(z) 
\right|^{p}dz\right]^{q/p} dx \right\}^{1/q}
\\
& = \left\{ \left( \int_D  \left[ \int_D \beta^{p}(x) \mathbf{1}_{tx+(1-t)D}(z)
\left|d\omega(z) \right|^{p} dz \right]^{q/p} dx \right)^{p/q} \right\}^{1/p}
\\
& = \left\{ \left\| \int_D \beta^{p}(x) \mathbf{1}_{tx+(1-t)D}(z)
\left|d\omega(z) \right|^{p} dz \right\|_{L^{q/p}(D,dx)} \right\}^{1/p}
\\
& \le \left\{ \int_D \left\| \beta^{p}(x) \mathbf{1}_{tx+(1-t)D}(z)
\left|d\omega(z)\right|^{p} \right\|_{L^{q/p}(D,dx)} dz \right\}^{1/p}
\\
& = \left\{ \int_D \left( \int_D \beta^q(x) \mathbf{1}_{tx+(1-t)D}(z) 
\left|d\omega(z)\right|^{q} dx \right)^{p/q} dz \right\}^{1/p}
\\
& = \left\{ \int_D \left( \int_D \beta^q(x) \mathbf{1}_{tx+(1-t)D}(z) dx \right)^{p/q} 
\left|d\omega(z)\right|^{p} dz \right\}^{1/p}
\\
& \le \left( \sup_{z\in D} \int_D \beta^q(x) \mathbf{1}_{tx+(1-t)D}(z) dx \right)^{1/q}
\left( \int_D \left| d\omega(z) \right|^p dz \right)^{1/p}
\\
& = \sup_{z\in D} \| \beta(x) \mathbf{1}_{tx+(1-t)D}(z) \|_{L^q(D,dx)}\,
\left\|d\omega\right\|_{L^p(D,\Lambda^{k+1})}. 
\end{align*}
The proposition follows.
\end{proof}

\smallskip

Estimate 
\[
C(k,p,q,n,\beta) =\int_{0}^{1}  
\sup_{z\in D} \| \beta(x) \mathbf{1}_{tx+(1-t)D}(z) \|_{L^q(D,dx)}
t^k (1-t)^{-n/p} dt
\]
in particular cases. 

\begin{cor}
Suppose that $D$ is a~convex set of~finite measure in~$\mathbb{R}^n$, $q\ge p\ge 1$,
$\frac{1}{p}-\frac{1}{q} < \frac{1}{n}$, and the~weight $\beta(x)\equiv 1$. Then
\[
C(k,p,q,n,1)\le |D|^{1/q}\int_{0}^{1}t^{k-n/q}(1-t)^{-n/p}\min(t^{n/q},(1-t)^{n/q})dt.
\]
\end{cor}

\begin{rem}
It is easy to see that the~integral of~the~corollary exists because of~the~conditions
imposed on~$p$ and~$q$.
\end{rem}

\begin{proof}
Using the change of variables $u=tx$, we obtain
\begin{multline*}
\int_{0}^{1} \sup_{z\in D}\left\| \mathbf{1}_{tx+(1-t)D}(z)\right\|_{L^q(D,dx)} t^{k}(1-t)^{-n/p}dt
\\
=\int_{0}^{1}\sup_{z\in D}\left\| \mathbf{1}_{u+(1-t)D}(z) 
\right\|_{L^{q}(tD,du)}t^{k-n/q}(1-t)^{-n/p}dt.
\end{multline*}

Note that $\left|tD\cap\left\{ u+(1-t)D\right\} \right|\leq\left|D\right|\min(t^{n},(1-t)^{n})$.
It follows that 
\begin{gather*}
\left\| \mathbf{1}_{u+(1-t)D}(z)\right\|_{L^q(tD,du)}\leq 
|D|^{1/q} \min(t^{n/q},(1-t)^{n/q});
\\
C(k,p,q,n,1)\leq |D|^{1/q}\int_{0}^{1}t^{k-n/q}(1-t)^{-n/p}\min(t^{n/q},(1-t)^{n/q})dt
\end{gather*}
\end{proof}

\begin{cor}\label{cor:beta}
Suppose that $U$ is a~convex set of finite measure 
$|U|$ in $\mathbb{R}^n$, $D=[a,b)\times U$, $q\ge p\ge 1$, 
$\frac{1}{p}-\frac{1}{q}< \frac{q-1}{q(n+1)}$,
and $\beta:[a,b)\to \mathbb{R}$ is an~integrable positive function. If 
$\left\| \beta \right\|_{L^{q}([a,b))}< \infty$ then 
\[
C(k,p,q,n,\beta)\leq\left|U\right|^{1/q}\left\| \beta \right\|_{L^{q}([a,b))}.
\]
\end{cor}

\begin{proof}
If $x\in D$ then $x=(\tau,w)$, where $\tau\in [a,b)$ and $w\in U$. 
Using the special type of the weight $\beta(x):=\beta(\tau)$ and representing $z\in D$
as $z=(\eta,\zeta)$ with $\eta\in [a,b)$ and $\zeta\in U$, we obtain
\begin{multline*}
\int_{0}^{1} \sup_{z\in D}
\left\| \beta(x)\mathbf{1}_{tx+(1-t)D}(z)\right\|_{L^{q}(D,dx)}t^{k}(1-t)^{-\frac{n+1}{p}}dt
\\
\leq \!\! \int_{0}^{1}\!\! \sup_{a \leq \eta <b}  \!\left(\!\int_{a}^{b}\! \beta^q (\tau)
\mathbf{1}_{t\tau+(1-t)[a,b)}(\eta) d\tau\! \right)^{\frac{1}{q}} \!\!\!
\sup_{\zeta\in U} \!\! \left( \int_{U} \! \mathbf{1}_{tw+(1-t)U}(\zeta) 
dw \! \right)^{\frac{1}{q}} \!\! t^{k}(1-t)^{-\frac{n+1}{p}}dt, 
\end{multline*}
where $x=(\tau,w)$.

Using the change of variables $u=tw$ and the estimate 
$$
\left|tU\cap\{ u+(1-t)U\} \right| \leq |U| \min(t^{n},(1-t)^{n}),
$$
we finally get
\begin{multline*}
\int_{0}^{1}\!\! \sup_{a \leq \eta <b}  \!\left(\!\int_{a}^{b}\! \beta^q (\tau)
\mathbf{1}_{t\tau+(1-t)[a,b)}(\eta) d\tau\! \right)^{\frac{1}{q}} \!\!\!
\sup_{\zeta\in U} \!\! \left( \int_{U} \! \mathbf{1}_{tw+(1-t)U}(\zeta) 
dw \! \right)^{\frac{1}{q}} \!\! t^{k}(1-t)^{-\frac{n+1}{p}}dt
\\
\leq |U|^{1/q}\|\beta\|_{L^{q}([a,b))}
\int_{0}^{1} t^{k-n/q}(1-t)^{-(n+1)/p}\min(t^{n/q},(1-t)^{n/q})dt
\end{multline*}

The~conditions on~$p$ and~$q$ imply the~finiteness of~the~last integral.
\end{proof}

Corollary~\ref{cor:beta} is a~key ingredient in~the~proof of~out main result, 
Theorem~\ref{thm: main global}. Unfortunately, for being able to~``separate'' 
the~variable~$t$, we have to~impose the~stronger constraint 
$\frac{1}{p}-\frac{1}{q}< \frac{1}{n+1}-\frac{1}{q(n+1)}$ than the~condition
$\frac{1}{p}-\frac{1}{q}< \frac{1}{n+1}$ given by~Proposition~\ref{bounded}.

\section{A New Homotopy Operator for $q\ge p$. \\ The Case of 
a~Convex Domain in~$\mathbb{R}^n$}\label{new-homot}

In the previous section, we considered the~homotopy operator on~$\Omega_{\loc}^*$
of~the~form
$$
A_\alpha=\int_{D}\alpha(y)K_{y}\omega(x)dy
$$
for a~convex set~$D$ in~$\mathbb{R}^n$. We will need to~modify~$A$ for obtaining
some estimates. 

Consider the same operator $K_{y}$ as in~the~previous section:
\[
\psi_{y}(x,t)=tx+(1-t)y,\quad K_{y}\omega(x)=\int_{0}^{1}(\psi_{y})_1^{*}\omega dt.
\]
Recall that $dK_{y}\omega+K_{y}d\omega=\omega$. Choose a~smooth positive function 
$\alpha:D\to \mathbb{R}$ such that $\int_D \alpha(x)dx=1$ and put
\[
A_{\alpha}\omega(x):=\int_{D}\alpha(y)K_{y}\omega(x)dy, \quad \omega\in \Omega_{\loc}^*.
\]

By a~straightforward calculation,
\[
dA_{\alpha}\omega=d\left(\int_{D}\alpha(y)K_{y}\omega(x)dy\right)
=\int_{D}\alpha(y)d_{x}K_{y}\omega(x)dy;\]
\[
A_{\alpha}d\omega=\int_{D}\alpha(y)K_{y}d\omega(x)dy;\]
\[
dA_{\alpha}\omega+A_{\alpha}d\omega
=\int_{D}\alpha(y)\left[d_{x}K_{y}\omega(x)+K_{y}d\omega(x)\right]dy
=\int_{D}\alpha(y)\omega(x)dy=\omega.
\]
In particular, if $d\omega=0$ then 
$$
dA_\alpha\omega =\omega.
$$

The~definition of~$A_\alpha$ easily implies the~following

\begin{prop}\label{homot-sm}
The~homotopy operator~$A_\alpha$ takes smooth forms to~smooth forms.
\end{prop}

\begin{defn}
Call a~smooth positive function $\alpha:D\to \mathbb{R}$ an {\em admissible weight} 
for a~convex domain $D\subset \mathbb{R}^n$ and $p\geq 1$ if 
$$
\int_D \alpha(x)dx=1; \quad \|\alpha\|_{L^{p'}(D)} <\infty; 
\quad \|\alpha(y)|y|\,\|_{L^{p'}(D)} <\infty.
$$
\end{defn}

For $p\ge 1$, we as usual put 
$$
p' = \begin{cases} 
\frac{p}{p-1} & \text{~if~} p>1, \\
\infty        & \text{~if~} p=1
\end{cases} 
$$

\begin{thm} \label{thm:one weight} 
Suppose that $q\ge p\geq 1$,
$D\subset \mathbb{R}^n$ is a~convex set, $\beta:D\to \mathbb{R}$ 
is a positive smooth function, and $\alpha:D\to \mathbb{R}$ is an~admissible weight. If 
$$
C_{1}(k,p,q,n,\beta):=\int_{0}^{1} 
\sup_{z\in D} \| \beta(x) \mathbf{1}_{tx+(1-t)D}(z) \|_{L^q(D,dx)}
t^k (1-t)^{-n/p} dt <\infty;
$$
$$
C_{2}(k,p,q,n,\beta):=\int_{0}^{1} 
\sup_{z\in D} \| |x|\beta(x) \mathbf{1}_{tx+(1-t)D}(z) \|_{L^q(D,dx)}
t^k (1-t)^{-n/p} dt < \infty
$$
then for any $\omega\in C^\infty L^{p}(D,\Lambda^{k})$ we have
\[
\| A_\alpha \omega\| _{L^{q}(D,\Lambda^{k-1},\beta)}\
\leq C(k,p,q,\alpha,\beta,n)\| \omega\| _{L^{p}(D,\Lambda^{k})}
\]
where 
$$
C(k,p,q,\alpha,\beta,n)= \| \alpha(y)|y|\| _{L^{p'}(D)} C_{1}(k,p,q,n,\beta)
+\|\alpha\|_{L^{p'}(D)} C_{2}(k,p,q,n,\beta).
$$
\end{thm}

\begin{proof}
Put $\xi:=A_{\alpha}\omega$. If $p>1$ then, by H\"older's inequality, we infer
\begin{multline*}
\left\| A_{\alpha}\omega\right\| _{L^{q}(D,\Lambda^{k-1},\beta)}
=\left\| \beta(x)\int_{D}\alpha(y)K_{y}\omega(x)dy\right\| _{L^{q}(D,\Lambda^{k-1},dx)}
\\
\leq
\left\| \beta(x)\left\| \frac{K_{y}\omega(x)}{|x-y|}
\right\| _{L^{p}(D,\Lambda^{k-1},dy)}\left\| \alpha(y)|x-y|\right\| _{L^{p'}(D,dy)}
\right\| _{L^{q}(D,\Lambda^{k-1},dx)}.
\end{multline*}

The~above estimate also obviously holds for~$p=1$.

By the triangle inequality,
\[
\| \alpha(y)|x-y|\| _{L^{p'}(D,dy)}
\leq |x| \| \alpha(y)\| _{L^{p'}(D,dy)}+ \| \alpha(y)|y|\,\| _{L^{p'}(D,dy)}.
\]

Therefore,
\begin{multline*}
\| A_{\alpha}\omega\| _{L^{q}(\beta,D,\Lambda^{k-1})}
\\
\leq \| \alpha(y)|y| \,\| _{L^{p'}(D,dy)}
\left\| \beta(x)\left\| \frac{K_{y}\omega(x)}{|x-y|}\right\| _{L^{p}(D,\Lambda^{k-1},dy)}
\right\| _{L^{q}(D,\Lambda^{k-1},dx)}
\\
+\| \alpha(y)\| _{L^{p'}(D,dy)}
\left\| \beta(x)|x| \left\| \frac{K_{y}\omega(x)}{|x-y|}\right\| _{L^{p}(D,\Lambda^{k-1},dy)}
\right\| _{L^{q}(D,\Lambda^{k-1},dx)}.
\end{multline*}
By Proposition~\ref{bounded},
\begin{gather*}
\left\| \beta(x)\left\| \frac{K_{y}\omega(x)}{|x-y|}\right
\| _{L^{p}(D,\Lambda^{k-1},dy)}\right\| _{L^{q}(D,\Lambda^{k-1},dx)}
\leq C_{1}(k,p,q,n,\beta)\left\| \omega\right\| _{L^{p}(D,\Lambda^{k})};
\\
\left\| \beta(x)|x|\left\| \frac{K_{y}\omega(x)}{|x-y|}\right
\| _{L^{p}(D,\Lambda^{k-1},dy)}\right\| _{L^{q}(D,\Lambda^{k-1},dx)}
\leq C_{2}(k,p,q,n,\beta)\left\| \omega\right\| _{L^{p}(D,\Lambda^{k})}.
\end{gather*}

The theorem is proved.
\end{proof}

\begin{cor}\label{twow}
Suppose that $q\ge p\geq 1$, 
$D\subset \mathbb{R}^n$ is a~convex set, $\alpha:[a,b)\to\mathbb{R}$ 
is an~admissible weight,
$\beta,\gamma:D\to \mathbb{R}$ are positive smooth functions. If the conditions 
\begin{gather*}
C_{1}(k,\overline{p},q,n,\beta):=\int_{0}^{1} 
\sup_{z\in D} \| \beta(x) \mathbf{1}_{tx+(1-t)D}(z) \|_{L^q(D,dx)}
t^k (1-t)^{-n/\overline{p}} dt < \infty;
\\
C_{2}(k,\overline{p},q,n,\beta):=\int_{0}^{1} 
\sup_{z\in D} \| |x| \beta(x) \mathbf{1}_{tx+(1-t)D}(z) \|_{L^q(D,dx)}
t^k (1-t)^{-n/\overline{p}} dt < \infty;
\\
Q(k,\overline{p},p,\gamma):=\left\| \gamma^{-1}\right\| _{L^{p\overline{p}/(p-\overline{p})}(D)}<\infty
\end{gather*}
are fulfilled for some $\overline{p}$, $1\le \overline{p}\le p$ {\rm(}for $\overline{p}=p$, we put 
$\frac{p\overline{p}}{p-\overline{p}}=\infty${\rm)},
then the inequality
\[
\| A_\alpha \omega\| _{L^{q}(D,\Lambda^{k-1},\beta)}
\leq C(k,p,q,\alpha,\beta,\gamma,n)\| \omega\| _{L^{p}(D,\Lambda^{k},\gamma)},
\]
where 
\[
C(k,p,q,\alpha,\beta,\gamma,n)=Q(k,\overline{p},p,\gamma)\: C(k,\overline{p},q,n,\alpha,\beta),
\]
holds for any $\omega\in C^\infty L^{p}(D,\Lambda^{k})$.
\end{cor}

\begin{proof}
By Theorem \ref{thm:one weight}, 
\[
\| A_\alpha \omega \| _{L^{q}(D,\Lambda^{k-1},\beta)}
\leq C(k,\overline{p},q,n,\alpha,\beta)\left\| \omega\right\| _{L^{\overline{p}}(D,\Lambda^{k})}.
\]
If $\overline{p}<p$ then, using H\"older's inequality, we have 
\begin{equation}\label{est-gam}
\left\| \omega\right\| _{L^{\overline{p}}(D,\Lambda^{k})}
\leq\left\| \gamma\omega\right\| _{L^{p}(D,\Lambda^{k})}
\left\| \gamma^{-1}\right\| _{L^{p\overline{p}/(p-\overline{p})}(D)}.
\end{equation}
Inequality~(\ref{est-gam}) also holds for~$\overline{p}=p$.

The corollary follows.
\end{proof}

\begin{cor}\label{cor:two weights}
Suppose that $q \ge p\geq 1$, $\frac{1}{p}-\frac{1}{q}<\frac{q-1}{q(n+1)}$,
$U$ is a~bounded convex set in~$\mathbb{R}^n$, $D=[a,b)\times U$, 
$\alpha:[a,b)\to\mathbb{R}$ is an~admissible weight,
and $\beta,\gamma:[a,b)\to \mathbb{R}$ are positive smooth functions. If 
the conditions $\|\beta\|_{L^{q}([a,b))}\!\!<\!\!\infty$,
$\| \tau\beta(\tau)\|_{L^{q}([a,b))}\!\!<\!\!\infty$,
and $\|\gamma^{-1}\|_{L^{p\overline{p}/(p-\overline{p})}([a,b))}<\infty$ 
are fulfilled for some $\overline{p}$, $1\le \overline{p}\le p$ 
{\rm(}for $\overline{p}=p$, we put $\frac{p\overline{p}}{p-\overline{p}}=\infty${\rm)},
then the inequality
$$
\| A_\alpha \omega \| _{L^q(D,\Lambda^{k-1},\beta)}
\leq \const \|\omega\|_{L^p(D,\Lambda^k,\gamma)}
$$
with some constant depending $k$,$p$,$q$,$n$,$\alpha$,$\beta$, and $\gamma$
holds for any $\omega\in C^\infty L^{p}(D,\Lambda^{k},\gamma)$.
\end{cor}

\begin{proof}
Suppose that a~number~$\overline{p}\le p$ satisfies the~conditions of~the~corollary.

If $x\in D$ then $x=(\tau,w)$, where $\tau\in [a,b)$ and $w\in U$. 
By~Corollary~\ref{cor:beta}, since 
$\frac{1}{\overline{p}}-\frac{1}{q}<\frac{q-1}{q(n+1)}$ and 
$\|\beta\|_{L^{q}([a,b))}<\infty$, we have
\begin{multline*}
\int_{0}^{1} \sup_{z\in D}
\left\| \beta(\tau)\mathbf{1}_{tx+(1-t)D}(z)\right\|_{L^{q}(D,dx)}t^{k}
(1-t)^{-(n+1)/\overline{p}}dt
\\
\leq |U|^{1/q}\|\beta\|_{L^{q}([a,b))}
\int_{0}^{1} t^{k-n/q}(1-t)^{-(n+1)/p}\min(t^{n/q},(1-t)^{n/q})dt.
\end{multline*}
On~the~other hand, since $\| \tau\beta(\tau)\|_{L^{q}([a,b))}<\infty$, we have
by~Corollary~\ref{cor:beta}:
\begin{multline*}
\int_{0}^{1} \sup_{z\in D}
\left\|\, |x|\beta(\tau)\mathbf{1}_{tx+(1-t)D}(z)\right\|_{L^{q}(D,dx)}t^{k}
(1-t)^{-(n+1)/\overline{p}}dt
\\
= \int_{0}^{1} \sup_{z\in D}
\left\|\, \sqrt{\tau^2+w^2}\beta(\tau)\mathbf{1}_{tx+(1-t)D}(z)\right\|_{L^{q}(D,dx)}t^{k}
(1-t)^{-(n+1)/\overline{p}}dt
\\
\le \sqrt{2} \int_{0}^{1} \sup_{z\in D}
\left\|\, (|\tau|+|w|)\beta(\tau)\mathbf{1}_{tx+(1-t)D}(z)\right\|_{L^{q}(D,dx)}t^{k}
(1-t)^{-(n+1)/\overline{p}}dt
\\
\le \sqrt{2} \int_{0}^{1} \sup_{z\in D}
\left\|\, |\tau|\beta(\tau) \mathbf{1}_{tx+(1-t)D}(z)\right\|_{L^{q}(D,dx)}t^{k}
(1-t)^{-(n+1)/\overline{p}}dt
\\
+ \sqrt{2} \int_{0}^{1} \sup_{z\in D}
\left\|\, |w|\beta(\tau) \mathbf{1}_{tx+(1-t)D}(z)\right\|_{L^{q}(D,dx)}t^{k}
(1-t)^{-(n+1)/\overline{p}}dt
\\
\le \sqrt{2} |U|^{1/q} \|\tau\beta(\tau)\|_{L^{q}([a,b))}
\int_{0}^{1} t^{k-n/q}(1-t)^{-(n+1)/\overline{p}}\min(t^{n/q},(1-t)^{n/q})dt
\\
+ \sqrt{2} \sup_{w\in U}|w| \, \|\beta\|_{L^{q}([a,b))}
\int_{0}^{1} t^{k-n/q}(1-t)^{-(n+1)/\overline{p}}\min(t^{n/q},(1-t)^{n/q})dt<\infty.
\end{multline*}

The~relations $\| \tau\beta(\tau)\|_{L^{q}([a,b))}<\infty$ and
$\| |x|\beta(\tau)\|_{L^{q}(D)}<\infty$ enable us to~apply Corollary~\ref{twow}
and obtain the~desired assertion.
\end{proof}

 \section{Globalization: the~Sobolev--Poincare Inequality on~a~Cylinder}\label{global}

Here we globalize the~Sobolev--Poincar\'e inequality to~cylinders. The main assertion
of~the~section is

\begin{thm}\label{glob-sp-cyl}
Suppose that $M$ is the~cylinder $[a,b)\times N$, where $N$ is a~closed manifold 
of~dimension~$n$, $q\ge p\geq 1$, $\frac{1}{p}-\frac{1}{q}<\frac{q-1}{q(n+1)}$, and
$\beta, \gamma:[a,b)\to \mathbb{R}$ be positive smooth functions. Let $\omega$ 
be an~exact $k$-form in~$C^\infty L^p(M,\Lambda^k,\gamma)$. 
If the conditions $\|\beta\| _{L^{q}([a,b))}<\infty$,
$\| t\beta(t)\|_{L^{q}([a,b))}<\infty$,
and $\|\gamma^{-1}\|_{L^{p\overline{p}/(p-\overline{p})}([a,b))}<\infty$ 
are fulfilled for some $\overline{p}$, $1\le \overline{p}\le p$ 
{\rm(}for $\overline{p}=p$, we put $\frac{p\overline{p}}{p-\overline{p}}=\infty${\rm)}, 
then there exists a~$(k-1)$-form $\xi\in C^\infty L^q(M,\Lambda^{k-1},\beta)$ 
such that
\begin{equation}\label{poin-cyl}
d\xi=\omega \quad \text{and} \quad
\|\xi\|_{L^q(M,\Lambda^{k-1},\beta)} \le \const\,\|\omega\|_{L^p(M,\Lambda^k,\gamma)}.
\end{equation} 
\end{thm}

\smallskip
Let $\tilde{\mathcal{U}}=\{\tilde{U}_x\}$, $x\in N$, be a~coordinate open cover 
of~the~base~$N$. At~each point $x\in N$, consider a~geodesic ball $U_x$ that is
geodesically convex (small balls are geodesically convex, 
see~\cite[Proposition~4.2]{Carm92}) and such that its closure (a~compact set) is contained 
in~$\tilde{U}_x$. Then $\mathcal{U^0}=\{\tilde{U}_x\}$ is an~open cover 
of~$N$. Extract a~finite subcover $\mathcal{U}=\{U_i\}$, $i=1,\dots,l$, from~$U^0$. 
Since $\mathcal{U}$ consists of~geodesic balls, it is a~{\it good cover}, 
i.e., all finite intersections $U_I = U_{i_0}\cap\dots\cap U_{i_{s-1}}$, 
$I=(i_0,\dots,i_{s-1})$, are bi-Lipschitz diffeomorphic to~convex open sets 
with~compact closure in~$\mathbb{R}^n$. With such a~cover~$\mathcal{U}$, 
we associate the~corresponding cover
$\mathcal{V}=\{V_i=[a,b)\times U_i\}$, $i=1,\dots,l$, of~$M$ and put
$V_I = V_{i_0}\cap\dots\cap V_{i_{s-1}}$ for $I=(i_0,\dots,i_{s-1})$. Then each
intersection~$V_I$ is bi-Lipschitz diffeomorphic to~a~cylinder of~the~form
$[a,b)\times U_{\mathbb{R}^n}$, where $U_{\mathbb{R}^n}$ is a~convex set with~compact
closure in~$\mathbb{R}^n$. By~analogy with~\cite{Shar2011}, we put
$$
K^{k,0}:= C^\infty(M,\Lambda^k);  \quad 
K^{k,s}:= \bigoplus_{i_0<\dots<i_{s-1}} C^\infty(V_I,\Lambda^k).
$$ 
Given $\varkappa\in K^{r,s}$, denote by~$\varkappa_I$, $I=(i_0,\dots,i_{s-1})$, 
$i_0<\dots<i_s$, the~components of~$\varkappa$. Define a~coboundary operator
$\delta: K^{k,s} \to K^{k,s+1}$ as follows:
$$
(\delta\varkappa)_J =  \left. \left(  \sum_{r=0}^s (-1)^r 
\varkappa_{j_0 \dots \hat j_r \dots j_s} \right) \right|_{V_J}, \quad
J=(j_0,\dots,j_s).
$$
Let $L^q(K^{k,s})$ be the~space of~elements $\varkappa\in K^{k,s}$ 
with the~finite norm
$$
\|\varkappa\|_{L^q(K^{k,s},\beta)} 
= \sum_{i_0<\dots i_{s-1}} \|\varkappa_I\|_{L^q(V_I,\Lambda^k,\beta)}.
$$
As usual, if $\varkappa\in K^{k,s}$ has components $\varkappa_I$, $I=(i_0,\dots,i_{s-1})$, 
$i_0<\dots<i_s$, and $\nu$ is a~permutation of~the~set $\{0,\dots,s-1\}$ then
$\alpha_{\nu(I)}= \alpha_I \mathrm{sign}\nu$.

The~following proposition is a~modification for~our case 
of~\cite[Proposition~3.6]{Shar2011}, which is in~turn an~adaptation 
of~\cite[Propositions~8.3 and 8.5]{BottTu}. 

\begin{prop}\label{exac}
$(K^{k,\bullet},\delta)$ is an~exact complex. Moreover, if 
$\lambda\in L^q(K^{k,s+1},\beta)$ satisfies $\delta \lambda=0$ 
then there exists $\varkappa\in L^q(K^{k,s},\beta)$ such that 
$\lambda=\delta\varkappa$ and
\begin{itemize}
\item $\|\varkappa\|_{L^q(K^{k,s},\beta)} \le \const \|\lambda\|_{L^q(K^{k,s+1},\beta)}$
\item $\|d\varkappa\|_{L^q(K^{k+1,s},\beta)} 
\le \const \left( \|\lambda\|_{L^q(K^{k,s+1},\beta)}
+ \|d\lambda\|_{L^q(K^{k+1,s+1},\beta)} \right)$.
\end{itemize}
\end{prop} 

\begin{proof}
The~fact that $(K^{k,\bullet},\delta)$ is an exact~complex was established 
in~\cite[Propositions~8.3 and~8.5]{BottTu} but we will give the~standard argument 
for~completeness. 
If $\varkappa\in L^q(K^{k,s},\beta)$ then
\begin{multline*}
(\delta(\delta\varkappa))_{i_0\dotsi_{s+1}}
= \sum_r (-1)^i (\delta\varkappa)_{i_0\dots\hat i_r\dots i_{s+1}} \\
=\sum_{l<r} (-1)^r (-1)^l 
\varkappa_{i_0\dots\hat i_l \dots \hat i_r \dots i_{s+1}} 
+ \sum_{l<r} (-1)^r (-1)^{l-1} 
\varkappa_{i_0\dots\hat i_l \dots \hat i_r \dots i_{s+1}}=0.
\end{multline*}

Suppose that
$\lambda\in L^q(K^{k,s+1},\beta)$ is such that $\delta\lambda=0$. Let $\tilde\rho_j$
be a~partition of~unity subordinate to~the~cover $\{U_i\}$ of~$N$. Then the~functions 
$\rho_j: M \to \mathbb{R}$, $\rho_j(t,x)=\tilde\rho_j(x)$ for all
$(t,x)\in M=[a,b)\times N$, constitute a~partition of~unity subordinate 
to~the~cover $\{V_i\}$ of~$M$. Put
\begin{equation}\label{kappa}
\varkappa_{i_0 \dots i_{s-1}} :=  \sum_j \rho_j \lambda_{j i_0 \dots i_{s-1}}. 
\end{equation}
Show that $\delta\varkappa=\lambda$. 

We have
$$
(\delta\varkappa)_{i_0 \dots i_s} 
= \sum_r (-1)^r \varkappa_{i_0 \dots \hat i_r \dots i_s}
= \sum_{r,j} (-1)^r \rho_j \lambda_{j i_0\dots \hat i_r \dots i_s}. 
$$
Since $\lambda$ is a~cocycle, 
$$
(\delta\lambda)_{j i_0 \dots i_s} = \lambda_{i_0 \dots i_s}
+ \sum_r (-1)^{r+1} \lambda_{j i_0 \dots \hat i_r \dots i_s} =0
$$
Hence,
$$
(\delta\varkappa)_{i_0 \dots i_s} = \sum_j \rho_j \sum_r (-1)^r 
\lambda_{j i_0 \dots \hat i_r \dots i_s} 
= \sum_j \rho_j \lambda_{i_0 \dots i_s} = \lambda_{i_0 \dots i_s}. 
$$

Thus, $(K^{k,\bullet},\delta)$ is indeed an~exact complex.

The~element~$\varkappa$ defined by~(\ref{kappa}) admits the~estimates 
of~the~norms mentioned in~the~proposition.

Indeed, we infer
\begin{align*}
\|\varkappa\|_{L^q(K^{k,s},\beta)}  
& = \sum_{i_0<\dots<i_{s-1}} \biggl\| \sum_j \rho_j 
\lambda_{j i_0 \dots i_{s-1}} \biggr\|_{L^q(U_I)} \\
&\le \sum_{i_0<\dots<i_{s-1}}  
\sum_j \| \rho_j \lambda_{j i_0 \dots i_{s-1}} \|_{L^q(U_I)} \\
&\le \sum_{i_0<\dots<i_{s-1}}  
\sum_j \| \lambda_{j i_0 \dots i_{s-1}} \|_{L^q(U_{j,I})}
\le \|\lambda\|_{L^q(K^{k,s+1},\beta)},
\end{align*}
which gives the~first estimate of~the~proposition.

Let us prove the~second estimate. We have
$$
d\varkappa_{i_0 \dots i_{s-1}} = \sum_j 
(d\rho_j \wedge \lambda_{j i_0 \dots i_{s-1}} + \rho_j d \lambda_{j i_0 \dots i_{s-1}}).
$$
Therefore,
\begin{multline*}
\|\varkappa\|_{L^q(K^{k+1,s},\beta)}  
= \sum_{i_0<\dots<i_{s-1}} \biggl\| \sum_j
d\rho_j \wedge \lambda_{j i_0 \dots i_{s-1}} + \rho_j d \lambda_{j i_0 \dots i_{s-1}}
\biggr\|_{L^q(U_I)}  \\
 \le \sum_{i_0<\dots<i_{s-1}} \sum_j
\left( \| d\rho_j \wedge \lambda_{j i_0 \dots i_{s-1}}\|_{L^q(U_I)} 
+ \|\rho_j d \lambda_{j i_0 \dots i_{s-1}}\|_{L^q(U_I)} \right) \\
\le \const 
\sum_{i_0<\dots<i_{s-1}} \sum_j
\left( \| \lambda_{j i_0 \dots i_{s-1}}\|_{L^q(U_I)} 
+ \| d \lambda_{j i_0 \dots i_{s-1}}\|_{L^q(U_I)} \right) \\
= \const \left( \|\lambda\|_{L^q(K^{k,s+1},\beta)} + \|d\lambda\|_{L^q(K^{k+1,s+1},\beta)}
\right).
\end{multline*}
\end{proof}

\smallskip
Now, applying the~general scheme of~\cite{Shar2011}, we first construct some 
elements $\xi^s\in L^q(K^{k-s-1,s+1},\beta)$ and then elements
$x^s\in L^q(K^{k-s-1,s},\beta)$ such that 
$\xi=x^0\in C^\infty L^q(M,\Lambda^{k-1},\beta)$
is an~element satisfying the~claim of~Theorem~\ref{glob-sp-cyl}.
\smallskip

{\bf Construction of~the~elements $\xi^s\in L^q(K^{k-s-1,s+1},\beta)$.}

Put $\xi^{-1}=\omega$ and define (by~induction) $\xi^s$ by~setting its 
component~$(\xi^s)_I$ to~be a~solution to~the~equation
\begin{equation}\label{eq-xi}
d\xi^s_I = (\delta\xi^{s-1})_I
\end{equation}
in~$V_I$, $I=(i_0,\dots,i_s)$ such that 
\begin{equation}
\|\xi^s_I\|_{L^q(V_I,\Lambda^{k-s-1},\beta)} 
\le \const \|(\delta\xi^{s-1})_I\|_{L^q(V_I,\Lambda^{k-s},\beta)}
\end{equation}                           
for $0\le s \le k-1$.

Note that such a~solution always exists due to~the~local Sobolev--Poincar\'e 
inequality (Corollary~\ref{cor:two weights}) since $V_I$ is bi-Lipschitz 
diffeomorphic to~a~cylinder over a~convex subset in~$\mathbb{R}^{n}$ 
of~finite volume.

We have the~following estimate of~the~weighted $q$-norm of~$\xi^s$:

\begin{prop}\label{norm-xi}
If $I=(i_0,\dots,i_s)$ then
$$
\|\xi^s_I\|_{L^q(V_I,\Lambda^{k-s-1},\beta)} \le 
\const \|\omega\|_{L^q(M,\Lambda^k,\gamma)}.
$$
\end{prop}

\begin{proof}
Use induction on~$s$. For $s=0$, the~assertion follows from~the~local 
Sobolev--Poincar\'e inequality. Let now $s>0$. We infer
\begin{align*}
\|\xi^s_I \|_{L^q(V_I,\Lambda^{k-s-1},\beta)} 
& \le \const \|(\delta\xi^{s-1})_I\|_{L^q(V_I,\Lambda^{k-s},\beta)}
\\
&\le \const \sum_{r=0}^s 
\|\xi^{s-1}_{i_0\dots\hat i_r\dots i_{s-1}}\|_{L^q(V_I,\Lambda^{k-s},\beta)}
\\
&\le \const \sum_{r=0}^s 
\|\xi^{s-1}_{i_0\dots\hat i_r\dots i_{s-1}}\|_{L^q(V_{i_0\dots\hat i_r\dots i_{s-1}},
\Lambda^{k-s},\beta)}
\\
&\le \const \sum_{r=0}^s \|\omega\|_{L^p(M,\gamma)} \le \const \|\omega\|_{L^p(M,\gamma)}
\end{align*}
\end{proof}

Note that $\xi^{k-1}$ is a~collection of~$0$-forms satisfying the~condition
\hbox{$d\delta\xi^{k-1}=0$}. Thus, the~functions $(\delta\xi^{k-1})_I$ are constants
on~each set~$V_I$, $I=(i_0,\dots,i_k)$. The~global constant functions
$(\delta\xi^{k-1})_I$ on~$M$ belong to~$L^q(M,\beta)$ due to~the~hypotheses on~$\beta$. 

The~following assertion is Theorem~3.10 in~\cite{Shar2011}:

\begin{lem}\label{th-shartser}
There exists $c\in K^{0,k}$ with constant components~$c_I$, $I=(i_0,\dots,i_{k-1})$,
such that
$$
(\delta c)_I = \sum_{r=0}^k (-1)^r c_{i_0\dots \hat i_r \dots i_k} 
(\delta\xi^{k-1})_I, \quad I=(i_0,\dots,i_k).
$$
In~addition, there exist numbers $b_{I,L}\in\mathbb{R}$, $I=(i_0,\dots,i_{k-1})$,
$L=(i_0,\dots,i_k)$, such that
$$
c_I= \sum_L b_{I,L} (\delta\xi^{k-1})_L,
$$
where $b_{i,L}$ depend on~the~chosen cover~$\mathcal{U}$ of~$N$.
\end{lem}

We have

\begin{prop}\label{c-estim}
The~constants~$c_I$ of Lemma~\ref{th-shartser} satisfy the~estimate
$$
\|c_I\|_{L^q(V_I,\beta)} \le \const \|\omega\|_{L^p(M,\Lambda^k,\gamma)}
$$ 
\end{prop}

\begin{proof}
By~Lemma~\ref{th-shartser}, each $c_I$ is representable as
$c_I=\sum_L b_{I,L} (\delta\xi^{k-1})_L$. Hence,
$$
\|c_I\|_{L^q(V_I,\beta)} \le \sum_L |b_{I,L}| \|(\delta\xi^{k-1})_L\|_{L^q(V_I,\beta)}. 
$$
Since $(\delta\xi^{k-1})_L$ is a~globally defined constant function on~$M$ as
in~the~proof of~Proposition~\ref{norm-xi}, we have
\begin{multline*}
\|(\delta\xi^{k-1})_L\|_{L^q(V_I,\beta)} = 
\frac{\|\beta\|_{L^q([a,b))} (\vol(U_I))^{1/q}}{\|\beta\|_{L^q([a,b))} (\vol(U_L))^{1/q}}
\|(\delta\xi^{k-1})_L\|_{L^q(V_L,\beta)}
\\
= \frac{(\vol(U_I))^{1/q}} {(\vol(U_L))^{1/q}}
\|(\delta\xi^{k-1})_L\|_{L^q(V_L,\beta)}
\le \const \|\omega\|_{L^p(M,\Lambda^k,\gamma)}.
\end{multline*}
This gives the~estimate of~the~proposition.
\end{proof}

{\bf Construction of~the~elements $x^s\in L^q(K^{k-s-1,s},\beta)$.}

Let us now glue all the~forms $\xi^s$, $s-0,\dots, k-1$, into~a~global form~$\xi$ 
satisfying~(\ref{poin-cyl}). Construct by~induction elements 
$x^s\in L^q(K^{k-s-1,s},\beta)$, $s=k-1,\dots,1,0$, such that $\xi=x^0$ 
is a~desired form on~$M$.

Put $\tilde\xi^{k-1}_I = \xi^{k-1}_I - c_I$, where $c_I$ is as in~Lemma~\ref{th-shartser},
$I=(i_0,\dots,i_{k-1})$. We have $d\tilde\xi^{k-1}_I = d\xi^{k-1}_I$ and
$\delta\tilde\xi^{k-1}_I=0$. By~Proposition~\ref{exac}, there exists 
$x^{k-1}\in L^q(K^{0,k-1},\beta)$ such that $\delta x^{k-1} = \tilde\xi^{k-1}$ and
\begin{gather*}
\|dx^{k-1}\|_{L^q(K^{1,k-1},\beta)} 
\le \const \|\tilde\xi^{k-1}\|_{L^q(K^{0,k},\beta)},
\\
\|x^{k-1}\|_{L^q(K^{0,k-1},\beta)} 
\le \const \left(\|\tilde\xi^{k-1}\|_{L^q(K^{0,k},\beta)}
+ \|d\tilde\xi^{k-1}\|_{L^q(K^{1,k},\beta)}\right).
\end{gather*}
Propositions~\ref{norm-xi} and~\ref{c-estim} yield
\begin{equation}\label{xk-1}
\|x^{k-1}\|_{L^q(K^{0,k-1},\beta)} \le\const \|\omega\|_{L^p(M,\Lambda^k,\gamma)}
\end{equation}
and 
\begin{multline}\label{dxk-1}
\|dx^{k-1}\|_{L^q(K^{1,k-1},\beta)} \le \const \left( \|\omega\|_{L^p(M,\Lambda^k,\gamma)}
+ \|\delta\xi^{k-2}\|_{L^q(K^{1,k},\beta)} \right)
\\
\le\const \left( \|\omega\|_{L^p(M,\Lambda^k,\gamma)} + \|\xi^{k-2}\|_{L^q(K^{1,k-1},\beta)}
\right) \le \const \|\omega\|_{L^p(M,\Lambda^k,\gamma)}.
\end{multline}
Suppose that $x^{k-(r-1)}$ is already constructed. By~Proposition~\ref{exac}, there
exists $x^{k-r}$ such that
$$
\delta x^{k-r} = \xi^{k-r} - dx^{k-r+1},
$$
where
\begin{equation}\label{xk-r}
\|x^{k-r}\|_{L^q(K^{r-1,k-r},\beta)} \le\const
\|\xi^{k-r} - dx^{k-r+1}\|_{L^q(K^{r-1,k-r+1},\beta)}
\end{equation}
and
\begin{multline} \label{dxk-r}
\|dx^{k-r}\|_{L^q(K^{r,k-r},\beta)} 
\\
\le\const \left( \|\xi^{k-r} - dx^{k-r+1}\|_{L^q(K^{r-1,k-r+1},\beta)}
+ \|d\xi^{k-r}\|_{L^q(K^{r,k-r+1},\beta)} \right)
\\
\le \const \bigl( \|\xi^{k-r}\|_{L^q(K^{r-1,k-r+1},\beta)} 
+ \|dx^{k-r+1}\|_{L^q(K^{r-1,k-r+1},\beta)} 
\\
+  \|\delta\xi^{k-r-1}\|_{L^q(K^{r,k-r+1},\beta)} \bigr)
\\
\le \const \left( \|\omega\|_{L^p(M,\Lambda^k,\gamma)}
+ \|dx^{k-r+1}\|_{L^q(K^{r-1,k-r+1},\beta)} \right).   
\end{multline}
Here the~last inequality stems from~the~fact that
$$
\delta(\xi^{k-r} - dx^{k-r+1}) = \delta\delta x^{k-r} = 0.
$$

The above considerations imply the~following

\begin{prop}\label{estim-xs}
The~forms~$x^s$ admit the~estimates:

{\rm(1)} $\|x^{k-r}\|_{L^q(K^{r-1,k-r},\beta)} 
\le\const \|\omega\|_{L^p(M,\Lambda^k,\gamma)}$;

{\rm(2)} $\|dx^{k-r}\|_{L^q(K^{r,k-r},\beta)}
\le\const \|\omega\|_{L^p(M,\Lambda^k,\gamma)}$.
\end{prop}

\begin{proof}
Use induction on~$r$. For $r=1$, (1) and (2) are just estimates~(\ref{xk-1})
and~(\ref{dxk-1}). Assume that $r>1$. For~proving estimate~(2), observe that,
by~the~induction hypothesis and~(\ref{dxk-r}),
\begin{align*}
\|dx^{k-r}\|_{L^q(K^{r,k-r},\beta)} 
&\le  \const \left( \|\omega\|_{L^p(M,\Lambda^k,\gamma)}
+ \|dx^{k-r+1}\|_{L^q(K^{r-1,k-r+1},\beta)}   \right)
\\
&\le \const \|\omega\|_{L^p(M,\Lambda^k,\gamma)}.
\end{align*}
Now, Proposition~\ref{norm-xi} and estimates~(\ref{xk-r}) and~(2) yield
\begin{multline*}
\|x^{k-r}\|_{L^q(K^{r-1,k-r},\beta)} \le\const
\|\xi^{k-r} - dx^{k-r+1}\|_{L^q(K^{r-1,k-r+1},\beta)} 
\\
\le \const \left( \|\xi^{k-r}\|_{L^q(K^{r-1,k-r+1},\beta)} + \|dx^{k-r+1}\|_{L^q(K^{r-1,k-r+1},\beta)} 
\right)
\\
\le\const \|\omega\|_{L^p(M,\Lambda^k,\gamma)}.
\end{multline*}
\end{proof}

Finally, put $\xi= x^0$. Then $d\xi=\omega$. Indeed, we have
$$
\delta(\omega-dx^0)=\delta\omega - d\delta x^0 
= \delta\omega - d(\xi^0-dx^1) = \delta\omega - d\xi^0=0.
$$
Since $\delta(\omega-dx^0)_i=(\omega-dx^0)|_{V_i}$, we infer that $\omega=dx_0$
on~$M$. By~Proposition~\ref{estim-xs}, 
$$
\|\xi\|_{L^q(M,\Lambda^{k-1},\beta)} = \|x^0\|_{L^q(K^{k-1,0},\beta)} 
\le\const \|\omega\|_{L^p(M,\Lambda^k,\gamma)}.
$$

Theorem~\ref{glob-sp-cyl} is completely proved.

\section{$L_{q,p}$-Cohomology of~a~Twisted Cylinder}\label{rel-cohom}

\begin{thm}\label{thm: main global}
Suppose that $N$ is a closed manifold of~dimension~$n$, $H^k_{\mathrm{DR}}(N)=0$,
$q\ge p\geq 1$, and $\frac{1}{p}-\frac{1}{q}<\frac{q-1}{q(n+1)}$.
If 
\[
\|\max(F_{k-2,q},F_{k-1,q})\|_{L^{q}([a,b))}<\infty, \quad
\| t \max(F_{k-2,q},F_{k-1,q})(t))\|_{L^{q}([a,b))}<\infty
\]  
and 
\[
\|\{\min(f_{k-1,p},f_{k,p})\}^{-1}\|_{L^{\frac{p\overline{p}}{p-\overline{p}}}([a,b))}
<\infty
\]
for some $\overline{p}$, $1\le \overline{p}\le p$ {\rm(}for $\overline{p}=p$, we put 
$\frac{p\overline{p}}{p-\overline{p}}=\infty${\rm)}, then $H_{q,p}^{k}(C_{a,b}^h N)=0$.
\end{thm}

\begin{proof}
Let $\overline{M}$ be the~cylinder~$[a,b)\times N$ with the~usual product metric.
By~the~K\"unneth formula for the~de~Rham cohomology, we have
$$
H^k_{\mathrm{DR}}(\overline{M})=H^k_{\mathrm{DR}}(N)=0.
$$

Using expression~(\ref{eq:norm}) for~the~norm and the~definition of~$f_{l,p}$, we infer
\begin{multline}\label{min-fp}
\|\omega\|_{L^{p}(\overline{M},\Lambda^{k},\min(f_{k-1,p},f_{k,p}))}
\\
=\left[\!\int_{a}^{b} \{\min(f_{k-1,p}(t),f_{k,p}(t))\}^p \int_{N}(|\omega_{A}(t,x)|_{N}^{2}
+ |\omega_{B}(t,x)|_{N}^{2}\bigr)^{\frac{p}{2}} 
dxdt\!\right]^{\frac{1}{p}}
\\
\leq  
\left[\!\int_{a}^{b}\int_{N}\bigl(h^{2(\frac{n}{p}-k)}(t,x) |\omega_{A}(t,x)|_{N}^{2}\!
+ h^{2(\frac{n}{p}-k+1)}(t,x) |\omega_{B}(t,x)|_{N}^{2}\bigr)^{\frac{p}{2}} 
dxdt\!\right]^{\frac{1}{p}}
\\
= \|\omega\|_{L^{p}(C_{a,b}^h N,\Lambda^{k})}.
\end{multline}

Thus, $\omega\in C^\infty L^{p}(\overline{M},\Lambda^{k},\min(f_{k-1,p},f_{k,p}))$.
Since the~de~Rham cohomology 
$H^k_{\mathrm{DR}}(\overline{M})$ is trivial, $\omega$ is exact, and we can apply 
Theorem~\ref{glob-sp-cyl}, by~which there exists
$\xi\in C^\infty L^{q}(\overline{M},\Lambda^{k},\max(F_{k-2,q}(t),F_{k-1,q}(t))$ with
\begin{equation}\label{est-norm-pq}
\|\xi\|_{L^q(\overline{M},\Lambda^{k-1},\max(F_{k-2,q},F_{k-1,q}))}  
\leq \mathrm{const} \|\omega\|_{L^{p}(\overline{M},\Lambda^{k},\min(f_{k-1,p},f_{k,p}))}.
\end{equation}
For this form~$\xi$, we have
\begin{multline}\label{max-Fq}
\|\xi\|_{L^q(C_{a,b}^h N,\Lambda^{k-1})} 
\\
=\left[\!\int_{a}^{b}\int_{N}\bigl(h^{2(\frac{n}{q}-k+1)}(t,x) |\xi_{A}(t,x)|_{N}^{2}\!
+ h^{2(\frac{n}{q}-k+2)}(t,x) |\xi_{B}(t,x)|_{N}^{2}\bigr)^{\frac{q}{2}} 
dxdt\!\right]^{\frac{1}{q}}
\\
\leq \left[\!\int_{a}^{b} \{\max(F_{k-2,q}(t),F_{k-1,q}(t))\}^q \int_{N} (|\xi_{A}(t,x)|_{N}^{2}
+ |\xi_{B}(t,x)|_{N}^{2}\bigr)^{\frac{q}{2}} dxdt\!\right]^{\frac{1}{q}}
\\
= \|\xi\|_{L^q(\overline{M},\Lambda^{k-1},\max(F_{k-2,q},F_{k-1,q}))}.  
\end{multline}
Combining~\eqref{min-fp},\eqref{est-norm-pq}, and \eqref{max-Fq}, we obtain
$$
\|\xi\|_{L^q(C_{a,b}^h N,\Lambda^{k-1})}  \le \const
\|\omega\|_{L^{p}(C_{a,b}^h N,\Lambda^{k})}.
$$
Thus, $C^\infty H_{q,p}^{k}(C_{a,b}^h N)=0$, and hence, by~Theorem~\ref{sm-cohom},
also $H_{q,p}^{k}(C_{a,b}^h N)=0$.
\end{proof}

\section{$L_{q,p}$-Cohomology of~an~Asymptotic Twisted Cylinder}
\label{cohom-asymp}

Recall the following definition, given in~\cite{GKop16}:

\begin{defn}
We refer to a pair $(M,X)$ consisting of an~$m$-dimensional manifold $M$ and
an~$m$-dimensional compact submanifold~$X$ with boundary as an {\it asymptotic twisted cylinder}
$AC_{a,b}^{h}\partial X$ if $M\setminus X$ is bi-Lipschitz diffeomorphically equivalent 
to the~twisted cylinder $C_{a,b}^{h}\partial X$. 
\end{defn}

For asymptotic twisted cylinders, Theorem~\ref{thm: main global} gives:

\begin{thm}\label{thm:app-abs} 
Let $(M,X)=AC_{a,b}^{h}\partial X$ be an~asymptotic twisted cylinder with 
$\dim M =\dim X = m = n+1$. Assume that $q\ge p\geq 1$, 
$\frac{1}{p}-\frac{1}{q}<\frac{q-1}{qm}$, and $H_{\mathrm{DR}}^{k}(X)=0$. If 
\[
\|\max(F_{k-2,q},F_{k-1,q})\|_{L^{q}([a,b))}<\infty, \quad
\| t\max(F_{k-2,q},F_{k-1,q})(t))\|_{L^{q}([a,b))}<\infty
\]
and
\[
\|\{\min(f_{k-1,p},f_{k,p})\}^{-1}\|_{L^{\frac{p\overline{p}}{p-\overline{p}}}([a,b))}
<\infty,
\]
for some $\overline{p}$, $1\le \overline{p}\le p$ {\rm(}for $\overline{p}=p$, we put 
$\frac{p\overline{p}}{p-\overline{p}}=\infty${\rm)}, then $H_{q,p}^{k}(M)=0$.
\end{thm}

\begin{proof}
Since bi-Lipschitz diffeomorphisms preserve $L_{p_1}$ and $L_{p_2}$ and extension
by zero gives a~topological isomorphism between the~spaces 
$W_{p_1,p_2}(C_{a,b}^h \partial X)$ and \linebreak
$W_{p_1,p_2}(M)$ for all~$p_1,p_2$, we have a~topological isomorphism 
$$
H^*_{p_1,p_2}(M)\cong H^*_{p_1,p_2}(C_{a,b}^h \partial X)
$$
for all~$p_1$, $p_2$. The theorem now follows from Theorem~\ref{thm: main global}.
\end{proof}

\section{Examples}\label{exam}

Let us analyze the~conditions of~the~last theorems for comparatively simple cases. 
Suppose that $N$ is the $n$-dimensional sphere $S^n$. Then $H^k_{\mathrm{DR}}(N)=0$ 
for any $k\neq n$. By the~hypothesess of~the~theorems, $q\ge p\geq 1$
 and $\frac{1}{p}-\frac{1}{q}<\frac{q-1}{q(n+1)}$. Put 
$$
s(t):=\max_{x\in S^n} h(t,x) 
\quad \text{and} \quad g(t):= \min_{x\in S^n} h(t,x).
$$
Then, by~definition,
$$
I_{1,q,k}:=\max(F_{k-2,q},F_{k-1,q})
=\max(s^{\frac{n}{q}-k+2},s^{\frac{n}{q}-k+1}),
$$
$$
I_{2,q,k}(t):= t \max(F_{k-2,q},F_{k-1,q})(t))=t\max(s^{\frac{n}{q}-k+2}(t),
s^{\frac{n}{q}-k+1}(t))
$$  
and 
$$
I_{3,p,k}:=\{\min(f_{k-1,p},f_{k,p})\}^{-1}
=\{\min(g^{\frac{n}{p}-k+1},g^{\frac{n}{p}-k})\}^{-1}.
$$

By~the~hypotheses of~the~theorems, we must check the~integrability of these three functions
in~the~corresponding degrees under the~above-mentioned restrictions on~$p$ and~$q$.

Suppose for simplicity that $s(t)$ and $g(t)$ are smooth increasing functions tending
to~$\infty$ as~$t\to b-0$. Denote the~maximal integrability intervals for~$s^u$ 
and $g^v$ by~$(-\infty, \alpha)$ and $(-\infty, \beta)$, i.e $s^u$ is integrable 
on~$[a,b)$ for every $u<\alpha$ and is not integrable for every $u>\alpha$ and similarly
for~$g^v$. Let also $\alpha_1$ be the~supremum of~$\mu$ such that 
$t s^{\mu}(t)$ is integrable on~$[a,b)$. 

For this case $I_{1,q,k}=s^{\frac{n}{q}-k+2}$, $I_{2,q,k}(t)=t s^{\frac{n}{q}-k+2}(t)$,
and $I_{3,p,k}=g^{k-\frac{n}{p}}$.

The~conditions of~the~theorems are fulfilled if 
$$
\frac{n}{q}-k+2<\min(\alpha, \alpha_1), \quad \frac{n}{p}-k>-\beta. 
$$

Note that these inequalities cannot hold simultaneously if $b=\infty$. 
In~this case, $\alpha$, $\alpha_1$, and~$\beta$ are all negative, whence
$\frac{n}{p}-k>\frac{n}{q}-k+2$. We thus have $\frac{1}{p}-\frac{1}{q}>\frac{2}{n}$,
which contradicts the~hypotheses.

Examine more closely the~case of~$0\leq a <b<\infty$. The~function~$t$ is bounded,
and hence $\alpha=\alpha_1$. Therefore, the~inequalities for $I_{1,q,k}$ and $I_{3,p,k}$
can be combined into one inequality 
$$
\frac{k-2+\alpha}{n}<\frac{1}{q}\le\frac{1}{p}<\frac{k-\beta}{n}.
$$
It means that the~additional condition  
$k-2+\alpha \leq k-\beta$, i.e., $\alpha+\beta \leq 2$,
must be fulfilled. 

The last condition is $\frac{1}{p}-\frac{1}{q}<\frac{q-1}{q(n+1)}$,
i.e., $p\le q<\frac{np}{n+1-p}$.

Summarizing, we conclude that for known integrability limits $\alpha$ and $\beta$,
we need to check two simple conditions for $p$ and~$q$:
$$
\alpha+\beta \leq 2, \, \, p<q<\frac{np}{n+1-p}
$$
and the inequality 
$$
\frac{k-2+\alpha}{n}<\frac{1}{q}\le\frac{1}{p}<\frac{k-\beta}{n}.
$$
for the degree~$k$.

Under these conditions, the~cohomology of~the~warped product~$C_{[a,b)}^f S^n$
vanishes.

For example, if $f(t)=g(t)=(b-t)^{-2}$ then $\alpha=\beta=1/2$. For $p=2$ we have
$2\le q<2 \frac{n}{n-1}$.

The last inequality yields
\begin{equation}\label{ineq-exam}
\frac{k-3/2}{n}<\frac{1}{q} \le \frac{1}{2} < \frac{k-1/2}{n}.
\end{equation} 
Let $q$ be an~arbitrary number in~$(2,\frac{8}{3})$. Then the~second inequality 
$q<\frac{2n}{n-1}$ gives us the~constraint $n<q/(q-2)$. Since $q/(q-2)<4$, we can take 
$n=4$. We have $1/2 < (2k-1)/8$, i.e., $k>2$. For $k=3$, the~leftmost inequality gives 
us the~fulfilled condition $3/8 < 1/q$. 
Note that if $q=2$ then always $q<\frac{2n}{n-1}$. If $n$ is even and $k=\frac{n}{2}+1$
then all inequalities in~(\ref{ineq-exam}) are fulfilled. Thus, we have 
$$
H_{q,2}^3 \bigl(C_{[a,b)}^h S^4\bigr)= 0  
\quad \text{if~} q\in \left[2,\frac{8}{3}\right)
$$
and
$$
H_{2,2}^{l+1}(C_{[a,b)}^h S^{2l}\bigr)=0, \quad l\ge 2.
$$

\end{document}